\documentclass[reqno]{amsart}
\usepackage{amsmath, amssymb, amsthm, epsfig}
\usepackage{hyperref, latexsym}
\usepackage{url}
\usepackage[mathscr]{euscript}

\usepackage{color}
\usepackage{fullpage} 
\usepackage{setspace}

\def\today{\ifcase\month\or
  January\or February\or March\or April\or May\or June\or
  July\or August\or September\or October\or November\or December\fi
  \space\number\day, \number\year}

 \newtheorem{theorem}{Theorem}
  
 \newtheorem{lemma}[theorem]{Lemma}
 
 \newtheorem{corollary}[theorem]{Corollary}
 \theoremstyle{definition}

 \theoremstyle{remark}

 \newcommand{\hh}{\tfrac12}
 \newcommand{\ds}{\text{\rm d}s}

  \renewcommand{\d}{\text{\rm d}}

\newcommand{\re}{{\rm Re}\,}

\newcommand{\dd}{\,{\rm d}}
\newcommand{\meio}{\frac{1}{2}}
\newcommand{\logfeio}{\log\log C(t,\pi)^{3/m}}
\newcommand{\escolha}{\max\left\{\frac{\logfeio-5\log\logfeio}{(1+2\vartheta)\pi},1\right\}}
\newcommand{\functionaleq}{\Lambda(s,\pi)=\kappa\,\Lambda(1-s,\tilde\pi)}

\begin{document}
\title[On the argument of $L$-functions]{On the argument of $L$-functions}
\author[Carneiro and Finder]{Emanuel Carneiro and Renan Finder}
\address{IMPA - Instituto de Matematica Pura e Aplicada - Estrada Dona Castorina, 110, Rio de Janeiro, RJ, Brazil 22460-320}
\email{carneiro@impa.br}
\email{feliz@impa.br}

\allowdisplaybreaks
\numberwithin{equation}{section}

\maketitle

\begin{abstract}
For $L(\cdot,\pi)$ in a large class of $L$-functions, assuming the generalized Riemann hypothesis, we show an explicit bound for the function $S_1(t,\pi)=\frac{1}{\pi}\int_{1/2}^\infty\log|L(\sigma+it,\pi)|\dd\sigma$, expressed in terms of its analytic conductor. This enables us to give an alternative proof of the most recent (conditional) bound for $S(t,\pi)=\frac{1}{\pi}{\rm arg}\,L\left(\meio+it,\pi\right)$, which is the derivative of $S_1(\cdot,\pi)$ at $t$.
\end{abstract}

\section{Introduction}

Throughout this note, the notation $O(E)$ refers to a quantity whose absolute value is bounded by a universal constant times $E$. For any integrable function $h:\mathbb R\to\mathbb C$, its Fourier transform is defined as $\widehat h(\xi)=\int_{-\infty}^\infty h(x)e^{-2\pi ix\xi}\dd x$. We use the function $$\Gamma_\mathbb R(z)=\pi^{-z/2}\,\Gamma\left(\frac{z}{2}\right),$$ where $\Gamma$ is the meromorphic extension of $z\mapsto\int_0^\infty x^{z-1}e^{-x}\dd x$.

\subsection{Background}
Let $\zeta$ be the Riemann zeta-function and let $S(t)=\frac{1}{\pi}\arg\zeta(\meio+it)$, where the argument is obtained by continuous variation along the ray $\{s\in\mathbb C\,|\,{\rm Re}\,s\geq\meio\text{ and }{\rm Im}\,s=t\}$, starting from $0$ at infinity. For $t\ge1$, the number of zeros of $\zeta$ whose imaginary part is between $0$ and $t$ is
\begin{equation}\label{numberofzeros}
N(t)=\frac{t}{2\pi}\log\frac{t}{2\pi}-\frac{t}{2\pi}+\frac{7}{8}+S(t)+O\left(\frac{1}{t}\right),
\end{equation}
provided that $S(t)$ and $N(t)$ are defined in a consistent way when $t$ is the imaginary part of a zero of $\zeta$.

\smallskip

In his article \cite{L}, Littlewood considered the function 
$$S_1(t)=\frac{1}{\pi}\int_{1/2}^\infty\log|\zeta(\sigma+it)|\dd\sigma$$ 
and observed that 
$$\int_t^u S(v)\dd v=S_1(u)-S_1(t).$$
The proof of \cite[Theorem 9]{L} shows how it is possible to use \eqref{numberofzeros} to derive a bound for $S$ from a bound for $S_1$. The idea is that, since $N(t)$ is nondecreasing, $S(t)$ does not decrease faster than $-\frac{t}{2\pi}\log\frac{t}{2\pi}+\frac{t}{2\pi}$, therefore a large value of $S(t)$ would cause a large variation of $S_1$ near $t$. In the same article, he assumed the Riemann hypothesis to conclude that $$S(t)=O\left(\frac{\log t}{\log\log t}\right)$$ and $$S_1(t)=O\left(\frac{\log t}{(\log\log t)^2}\right)$$ for $t\geq 3$.

\smallskip

The works \cite{F, GG, RS} went further by finding numerical bounds for 
\begin{equation}\label{limsup_S}
\limsup_{t\to\infty}\left|S(t)\left(\frac{\log t}{\log\log t}\right)^{-1}\right|,
\end{equation}
while \cite{F2,KK} exhibit numerical bounds for 
\begin{equation}\label{limsup_S_1}
\limsup_{t\to\infty}\left|S_1(t)\left(\frac{\log t}{(\log\log t)^2}\right)^{-1}\right|.
\end{equation}
Ramachandra and Sankaranarayanan, in \cite{RS}, also remark that a bound of the same kind is true for Dirichlet $L$-functions, assuming the corresponding Riemann hypothesis. The article \cite{GG} introduces the use of extremal functions of exponential type in this problem.

\smallskip

Currently, the best conditional bounds for \eqref{limsup_S} and \eqref{limsup_S_1} are due to Carneiro, Chandee and Milinovich. In the article \cite{CCM} they showed that, if the Riemann hypothesis holds, $$S_1(t)=\frac{1}{4\pi}\log t-\frac{1}{\pi}\sum_\gamma f_1(t-\gamma)+O(1),$$ where the sum is over all $\gamma$ such that $\zeta\left(\meio+i\gamma\right)=0$ and
\begin{equation}\label{2}
f_1(x)=\meio\int_{1/2}^{3/2}\log\frac{1+x^2}{\left(\sigma-\meio\right)^2+x^2}\dd\sigma=1-x\arctan\left(\frac{1}{x}\right).
\end{equation}
Then, the tools of \cite{CLV} were used to find real entire minorants and majorants of exponential type for $f_1$, which allowed the use of the Guinand-Weil explicit formula. By this method they obtained \cite[Theorem 1]{CCM}
\begin{equation}\label{3}
S_1(t)\leq \frac{\pi}{48}\frac{\log t}{(\log\log t)^2}+O\left(\frac{\log t\log\log\log t}{(\log\log t)^3}\right)\end{equation}
and
\begin{equation}\label{4}
S_1(t)\geq-\frac{\pi}{24}\frac{\log t}{(\log\log t)^2}+O\left(\frac{\log t\log\log\log t}{(\log\log t)^3}\right).
\end{equation}
These inequalities bound any difference $S_1(u)-S_1(t)=\int_t^uS(v)\dd v$. Also, \eqref{numberofzeros} may be used to compare this difference to $S(t)$, and choosing appropriate values of $u$ yields the inequality \cite[Theorem 2]{CCM}
\begin{equation}\label{1}
|S(t)|\leq\frac{1}{4}\frac{\log t}{\log\log t}+O\left(\frac{\log t\log\log\log t}{(\log\log t)^2}\right)
\end{equation}
for sufficiently large $t$. A shorter proof of \eqref{1} was recently obtained in \cite[Theorem 1]{CCM2} using the classical Beurling-Selberg majorants and minorants of characteristic functions of intervals and exploiting the fact that $\zeta$ is self-dual (i.e. $\zeta(s) = \overline{\zeta(\overline{s})}$).

\subsection{L-functions}
We work with a meromorphic function $L(\cdot,\pi)$ \footnote{Our notation is motivated by $L$-functions arising from cuspidal automorphic representations $\pi$ of $GL(m)$ over a number field.} on $\mathbb C$ which meets the following requirements (for some positive integer $m$ and some $\vartheta\in[0,1]$). The examples include the Dirichlet $L$-functions $L(\cdot,\chi)$ for primitive characters $\chi$.

\begin{enumerate}
\item[(i)] There exists a sequence $\{\lambda_\pi(n)\}_{n\ge1}$ of complex numbers ($\lambda_\pi(1) =1$) such that the series $$\sum_{n=1}^\infty\frac{\lambda_\pi(n)}{n^s}$$ converges absolutely to $L(s,\pi)$ on $\{s\in\mathbb C \,|\,\text{Re}\,s>1\}$.

\smallskip

\item[(ii)] For each prime number $p$, there exist $\alpha_{1,\pi}(p),\alpha_{2,\pi}(p),\ldots,\alpha_{m,\pi}(p)$ in $\mathbb C$ such that $|\alpha_{j,\pi}(p)|\leq p^\vartheta$ and 
$$L(s,\pi)=\prod_p\prod_{j=1}^m\left(1-\frac{\alpha_{j,\pi}(p)}{p^s}\right)^{-1},$$
with absolute convergence on the half plane $\text{Re}\,s>1$.

\smallskip

\item[(iii)] For some positive integer $N$ and some complex numbers $\mu_1,\mu_2,\ldots,\mu_m$ whose real parts are greater than $-1$ and such that $\{\mu_1,\mu_2,\ldots,\mu_m\}=\{\overline{\mu_1},\overline{\mu_2},\ldots,\overline{\mu_m}\}$, the completed $L$-function 
$$\Lambda(s,\pi)=N^{s/2}\prod_{j=1}^m \Gamma_\mathbb R(s+\mu_j)L(s,\pi)$$
is a meromorphic function of order 1 that has no poles other than $0$ and $1$. The points $0$ and $1$ are poles with the same order $r(\pi)\in\{0,1,\ldots,m\}$. Furthermore, the function $\Lambda(s, \tilde\pi):=\overline{\Lambda(\overline s, \pi)}$ satisfies the functional equation
\begin{equation}\label{Intro_FE}
\functionaleq
\end{equation}
for some unitary complex number $\kappa$.
\end{enumerate}

Except for the assumption $r(\pi)\le m$, we are in the same framework as \cite[Chapter 5]{IK}, where many examples may be found.

\subsection{Main results}

The theorems we prove are analogues of \eqref{3}, \eqref{4} and \eqref{1} for $L$-functions. They are based on the {\it generalized Riemann hypothesis}, which asserts that $\Lambda(s,\pi)\neq0$ if ${\rm Re}\,s\neq\meio$. The function $$C(t,\pi)=N\prod_{j=1}^m(|it+\mu_j|+3),$$ called the {\it analytic conductor} of $L(\cdot,\pi)$, is used in their statements. Our first result is about the function 
$$S_1(t,\pi)=\frac{1}{\pi}\int_{1/2}^\infty\log|L(\sigma+it,\pi)|\dd\sigma.$$
\begin{theorem} \label{theorem1}
Let $L(\cdot,\pi)$ satisfy the generalized Riemann hypothesis. Then, for every real number $t$,
\begin{align*}
S_1(t,\pi)\leq \frac{(1+2\vartheta)^2\,\pi}{48}\frac{\log C(t,\pi)}{(\logfeio)^2} +O\left(\frac{\log C(t,\pi)\log\logfeio}{(\logfeio)^3}\right)
\end{align*}
and
\begin{align*}
S_1(t,\pi)\geq-\frac{(1+2\vartheta)^2\,\pi}{24}\frac{\log C(t,\pi)}{(\logfeio)^2}+O\left(\frac{\log C(t,\pi)\log\logfeio}{(\logfeio)^3}\right).
\end{align*}
\end{theorem}

\smallskip

If $t$ is not the imaginary part of a zero of $L\left(\cdot,\pi\right)$ and $t\neq0$, the argument function is defined by 
$$S(t,\pi)=-\frac{1}{\pi}\int_{1/2}^\infty{\rm Im}\frac{L'}{L}(\sigma+it,\pi)\dd\sigma.$$ 
Otherwise, it is 
$$S(t,\pi)=\lim_{\eta\to0}\frac{S(t+\eta,\pi)+S(t-\eta,\pi)}{2}.$$ 
We note that $S_1(t,\pi)$ is a primitive for the function $S(t,\pi)$ (details in Section \ref{Sec3} below). An extension of \eqref{1} to $L$-functions, with a good leading constant, was obtained by Carneiro, Chandee and Milinovich in \cite[Theorem 5]{CCM2}, via a direct approach using extremal majorants and minorants of exponential type for the odd function $f(x) = \arctan\left(\frac{2}{x}\right) - \frac{2x}{4+x^2}$, available in the framework of \cite{CL}. Here we give an alternative proof of this result, deriving it from our Theorem \ref{theorem1}.
\begin{theorem} \label{theorem2}
Let $L(\cdot,\pi)$ satisfy the generalized Riemann hypothesis. Then, for every real number $t$,
$$|S(t,\pi)|\le\frac{1+2\vartheta}{4}\frac{\log C(t,\pi)}{\logfeio}+O\left(\frac{\log C(t,\pi)\log\logfeio}{(\logfeio)^2}\right).$$
\end{theorem}

\smallskip

The previous result gives information about the distribution of the zeros of $L$-functions. An example is the following corollary, related to \cite[Corollary 1]{GG} and \cite[Theorem 7]{CCM2}.
\begin{corollary}\label{Cor3}
Let $L(\cdot,\pi)$ satisfy the generalized Riemann hypothesis.
\begin{enumerate}
\item[(i)] Let $m(\gamma, \pi)$ denote the multiplicity of the zero $\hh + i\gamma$ of $\Lambda(\cdot,\pi)$. Then,
$$m(\gamma, \pi) \leq \frac{1+2\vartheta}{2}\frac{\log C(\gamma,\pi)}{\log\log C(\gamma,\pi)^{3/m}}+O\left(\frac{\log C(\gamma,\pi)\log\log\log C(\gamma,\pi)^{3/m}}{(\log\log C(\gamma,\pi)^{3/m})^2}\right).$$

\smallskip

\item[(ii)] Let $\meio+i\gamma$ and $\meio+i\gamma'$ be consecutive zeros of $\Lambda(\cdot,\pi)$. Then $\gamma'-\gamma$ is bounded by some universal constant and if $C(\gamma,\pi)^{3/m}$ is sufficiently large, 
$$\gamma'-\gamma\le\frac{(1+2\vartheta)\pi}{\log\log C(\gamma,\pi)^{3/m}}+O\left(\frac{\log\log\log C(\gamma,\pi)^{3/m}}{(\log\log C(\gamma,\pi)^{3/m})^2}\right).$$
 \end{enumerate}
\end{corollary}

\smallskip

We make no attempt here to estimate the universal bound for the gap between consecutive zeros of our general class of $L$-functions. Such a gap has been estimated (for a slightly restricted class of $L$-functions) in \cite[Theorem 2.1]{Bo}. In addition to $S(\cdot, \pi)$ and $S_1(\cdot, \pi)$, the theory of extremal functions of exponential type can be used to provide upper bounds for the modulus of an $L$-function on the critical line. This has been carried out in the work of Chandee and Soundararajan \cite{CS}:
\begin{align*}
\log\left|L\left(\hh+it,\pi\right)\right|\leq \frac{(1+2\vartheta)\log 2}{2}\frac{\log C(t,\pi)}{\logfeio} +O\left(\frac{\log C(t,\pi)\log\logfeio}{(\logfeio)^2}\right).
\end{align*}
Although they considered explicitly only the case $t=0$, their reasoning is general. Other examples of the use of bandlimited majorants to the theory of the Riemann zeta-function include \cite{CC,CCLM,G}.

\section{Proof of Theorem \ref{theorem1}}

In this section we prove Theorem \ref{theorem1}. We adapt the strategy of \cite{CCM}, where the case of the Riemann zeta-function was considered.

\begin{lemma} \label{lemma1}
Let $L(\cdot,\pi)$ satisfy the generalized Riemann hypothesis. For any real $t$, 
$$S_1(t,\pi)=\frac{1}{\pi}\left(-\sum_\gamma F_1(t-\gamma)+\log C(t,\pi)\right)+O(m),$$ 
where the sum is over all values of $\gamma$ such that $L\left(\meio+i\gamma,\pi\right)=0$, counted with multiplicity, and
\begin{equation}\label{5}
F_1(x)=\meio\int_{1/2}^{5/2}\log\frac{4+x^2}{(\sigma-\meio)^2+x^2}\dd\sigma=2-x\arctan\left(\frac{2}{x}\right).
\end{equation}
\end{lemma}
\begin{proof}
By the product expansion of $L(\cdot,\pi)$ and the inequality $|\alpha_{j,\pi}(p)|\leq p$, $$|\log|L(s,\pi)||\leq m\log\zeta({\rm Re}\,s-1)=O\left(\frac{m}{2^{{\rm Re}\,s}}\right)$$ for any $s$ such that ${\rm Re}\,s\geq\frac{5}{2}$. Because of this and of the fact that $L(\cdot,\pi)$ is meromorphic, $S_1(\cdot,\pi)$ is well defined and
\begin{align*}
\pi S_1(t,\pi)=&\int_{1/2}^{5/2}\log|L(\sigma+it,\pi)|\dd\sigma+O(m)\\
=&\int_{1/2}^{5/2}\left\{\log|L(\sigma+it,\pi)|-\log\left|L\left(\tfrac{5}{2}+it,\pi\right)\right|\right\}\d\sigma+O(m)\\
=&\int_{1/2}^{5/2}\left\{\log|\Lambda(\sigma+it,\pi)|-\log\left|\Lambda\left(\tfrac{5}{2}+it,\pi\right)\right|\right\}\d\sigma +\int_{1/2}^{5/2}\left\{\log\big|N^{(5/2+it)/2}\big|-\log\big|N^{(\sigma+it)/2}\big|\right\}\d\sigma\\
&\ \ \ \ \ \ \ \ +\sum_{j=1}^m\int_{1/2}^{5/2}\left\{\log\left|\Gamma_\mathbb R\left(\tfrac{5}{2}+it+\mu_j\right)\right|-\log|\Gamma_\mathbb R(\sigma+it+\mu_j)|\right\}\d\sigma +O(m).
\end{align*}
We treat each integral separately. For the first one, we use Hadamard's factorization formula 
$$\Lambda(s,\pi)=s^{-r(\pi)}(s-1)^{-r(\pi)}e^{A+Bs}\prod_{\rho}\left(1-\frac{s}{\rho}\right)e^{s/\rho},$$
where $A$ and $B$ are constants and the product is over all zeros of $\Lambda(\cdot,\pi)$. From the functional equation $\functionaleq$, one deduces that $\text{Re}\,B=-\sum_\rho\text{Re}\left(\frac{1}{\rho}\right)$ (see \cite[Proposition 5.7]{IK}). Hence, for $\meio\le\sigma\le\frac{5}{2}$, 
$$\left|\frac{\Lambda(\sigma+it,\pi)}{\Lambda\left(\frac{5}{2}+it,\pi\right)}\right|=\left|\frac{\sigma+it}{\frac{5}{2}+it}\right|^{-r(\pi)}\left|\frac{\sigma-1+it}{\frac{3}{2}+it}\right|^{-r(\pi)}\prod_{\rho}\left|\frac{\sigma+it-\rho}{\frac{5}{2}+it-\rho}\right|,$$ which implies, via the substitution $\rho=\meio+i\gamma$, that 
$$\log\left|\frac{\Lambda(\sigma+it,\pi)}{\Lambda\left(\frac{5}{2}+it,\pi\right)}\right|=O(m)+r(\pi)\log\left|\frac{\frac{3}{2}+it}{\sigma-1+it}\right|+\sum_{\gamma}\meio\log\frac{(\sigma-\meio)^2+(t-\gamma)^2}{4+(t-\gamma)^2}.$$ Since $1\leq\left|\frac{\frac{3}{2}+it}{\sigma-1+it}\right|\leq\frac{\frac{3}{2}}{|\sigma-1|}$, integrating we get 
$$\int_{1/2}^{5/2}\left\{\log\left|\Lambda(\sigma+it,\pi)\right|-\log\left|\Lambda\left(\tfrac{5}{2}+it,\pi\right)\right|\right\}\d\sigma=-\sum_\gamma F_1(t-\gamma)+O(m).$$ 

Our considerations on the last $m$ integrals use Stirling's formula $$\frac{\Gamma'}{\Gamma}(z)=\log(1+z)-\frac{1}{z}+O(1)$$ 
in the form 
\begin{equation}\label{Stirling}
\frac{\Gamma_\mathbb R'}{\Gamma_\mathbb R}(z)=\frac{1}{2}\log(2+z)-\frac{1}{z}+O(1),
\end{equation} 
valid for ${\rm Re}\,z>-\meio$. For any $\mu$ such that ${\rm Re}\,\mu>0$, integration by parts yields
\begin{align*}
\int_{1/2}^{5/2}& \left\{\log\left|\Gamma_\mathbb R\left(\tfrac{5}{2}+\mu+it\right)\right|-\log\left|\Gamma_\mathbb R(\sigma+\mu+it)\right|\right\}\d\sigma\\
=&\int_{1/2}^{5/2}\left(\sigma-\hh\right){\rm Re}\,\frac{\Gamma_\mathbb R'}{\Gamma_\mathbb R}(\sigma+\mu+it)\dd \sigma\\
=&\int_{1/2}^{5/2}\left(\sigma-\hh\right)\left(\hh\log|2+\sigma+\mu+it|-{\rm Re}\,\frac{1}{\sigma+\mu+it}\right)\d \sigma+O(1)\\
=&\meio\int_{1/2}^{5/2}\left(\sigma-\hh\right)\log|2+\sigma+\mu+it|\dd \sigma+O(1)\\
=&\meio\int_{1/2}^{5/2}\left(\sigma-\hh\right)\log(|\mu+it|+3)\dd\sigma+O(1)\\
=&\log(|\mu+it|+3)+O(1).
\end{align*}
If $-1<{\rm Re}\mu\le0$, the relation $\Gamma_\mathbb R(z+2)=\frac{z}{2\pi}\Gamma_\mathbb R(z)$ brings us back to the previous case. Indeed,
\begin{align*}
\int_{1/2}^{5/2} & \log\left\{\left|\Gamma_\mathbb R\left(\tfrac{5}{2}+\mu+it\right)\right|-\log\left|\Gamma_\mathbb R(\sigma+\mu+it)\right|\right\}\d\sigma\\
=&\int_{1/2}^{5/2}\left\{\log\left|\Gamma_\mathbb R\left(\tfrac{9}{2}+\mu+it\right)\right|-\log\left|\Gamma_\mathbb R(2+\sigma+\mu+it)\right|-\log\left|\frac{\frac{5}{2}+\mu+it}{\sigma+\mu+it}\right|\right\}\d\sigma\\
=&\log(|2+\mu+it|+3)+O(1)+O\left(\int_{1/2}^{5/2}\log\left|\frac{\frac{5}{2}+{\rm Re}\,\mu}{\sigma+{\rm Re}\,\mu}\right|\dd\sigma\right)\\
=&\log(|\mu+it|+3)+O(1),
\end{align*}
as before.

Finally, $\int_{1/2}^{5/2}\log|N^{(5/2+it)/2}|-\log|N^{(\sigma+it)/2}|\dd\sigma=\log N$. Combining our computations we get
\begin{align*}
\pi S_1(t,\pi)=&-\sum_\gamma F_1(t-\gamma)+\log N+\sum_{j=1}^m\log(|\mu_j+it|+3)+O(m)\\
=&-\sum_\gamma F_1(t-\gamma)+\log C(t,\pi)+O(m).
\end{align*}
\end{proof}
To estimate the infinite sum that appears in the preceding lemma, we employ the Guinand-Weil explicit formula. Its statement depends on noting that, by the product expansion of $L(\cdot,\pi)$, 
$$\frac{L'}{L}(s,\pi)=-\sum_p\sum_{j=1}^m\frac{\alpha_{j,\pi}(p)}{p^s}\left(1-\frac{\alpha_{j,\pi}(p)}{p^s}\right)^{-1}\log p\,,$$
where the right-hand side converges absolutely if ${\rm Re}\,s>1$. This shows that the logarithmic derivative of $L(\cdot,\pi)$ has a Dirichlet series: 
\begin{equation}\label{Def_Lambda}
\frac{L'}{L}(s,\pi)=-\sum_{n=2}^\infty\frac{\Lambda_\pi(n)}{n^s},
\end{equation}
where $\Lambda_\pi(n)=0$ if $n$ is not a power of prime and $\Lambda_\pi(p^k)=\sum_{j=1}^m\alpha_{j,\pi}(p)^k\log p$ if $p$ is prime and $k$ is a positive integer.

\begin{lemma} \label{lemma2}
Let $h$ be an analytic function defined on a strip $\{z\in\mathbb C\,|\,-\meio-\varepsilon<{\rm Im}\,z<\meio+\varepsilon\}$ such that $h(z)(1+|z|)^{1+\delta}$ is bounded for some positive $\delta$. Then 
\begin{align}\label{7}
\begin{split}
\sum_{\rho} h\left(\frac{\rho - \tfrac12}{i}\right)&= r(\pi)\left\{h\left(\frac{1}{2i}\right)+h\left(-\frac{1}{2i}\right)\right\} + \frac{\log N}{2\pi}\int_{-\infty}^\infty h(u)\dd u\\
& +\frac{1}{\pi}\sum_{j=1}^m\int_{-\infty}^\infty h(u)\,{\rm Re}\,\frac{\Gamma_\mathbb R'}{\Gamma_\mathbb R}\left(\hh+\mu_j+iu\right)\d u\\
& -\frac{1}{2\pi}\sum_{n=2}^\infty\frac{1}{\sqrt{n}}\left\{\Lambda_{\pi}(n)\, \widehat h\left(\frac{\log n}{2\pi}\right)+\overline{\Lambda_{\pi}(n)}\, \widehat h\left(\frac{-\log n}{2\pi}\right)\right\} \\
&- \!\!\sum_{-1 < {\rm Re}\,\mu_j<-\meio}\!\!\left\{h\left(\frac{-\mu_j-\meio}{i}\right)\!+h\left(\frac{\mu_j+\meio}{i}\right)\!\right\} -\meio\sum_{{\rm Re}\,\mu_j=-\meio} \left\{h\left(\frac{-\mu_j-\meio}{i}\right)\!+h\left(\frac{\mu_j+\meio}{i}\right)\!\right\},
\end{split}
\end{align}
where the sum runs over all zeros $\rho$ of $\Lambda(\cdot, \pi)$ and the coefficients $\Lambda_\pi(n)$ are defined by \eqref{Def_Lambda}.
\end{lemma}

\begin{proof}
This is a modification of the proof of \cite[Theorem 5.12]{IK}. The idea is to consider the integral
$$\frac{1}{2\pi i } \oint h\!\left(\frac{s- \tfrac12}{i}\right) \frac{\Lambda'(s,\pi)}{\Lambda(s, \pi)}\,\ds$$
over the rectangular contour connecting the points $1 + \eta + iT_1, - \eta + iT_1, - \eta - iT_2, 1 + \eta - iT_2$, say with $\eta = \varepsilon/2$. Then one sends $T_1, T_2 \to \infty$ over an appropriate sequence of heights that keep the zeros as far as possible (recall that at height $T$, we have $O(\log C(t,\pi))$ zeros, see \cite[Proposition 5.7]{IK}). One then uses the functional equation \eqref{Intro_FE} to replace the integral over the line $\re s = -\eta$ by an integral over the line $\re s = 1 + \eta$, and finally one moves the remaining integrals to the line $\re s = \hh$, picking up possibly some additional poles at the $\mu_j$'s.
\end{proof}

The function $F_1$ defined in \eqref{5} is not analytic. However, if we take $h$ equal to an analytic minorant or majorant of $F_1$, we obtain lower and upper bounds for $\sum_\gamma F_1(t-\gamma)$. To estimate the right-hand side of \eqref{7}, it is convenient to choose $\widehat h$ compactly supported and to minimize the $L^1$-norm of $h-F_1$. Carneiro, Littmann and Vaaler studied this minimization problem in a more abstract setting (see \cite{CLV}) and it was shown in \cite{CCM} that their result could be applied to the function $f_1$ defined by \eqref{2}. The following lemma is a rescaling of the obtained conclusion.

\begin{lemma} \label{lemma3}
For every $\Delta\geq1$, there is a unique pair of real entire functions $G_\Delta^-:\mathbb C\to\mathbb C$ and $G_\Delta^+:\mathbb C\to\mathbb C$ satisfying the following properties:
\begin{enumerate}
\item[(i)] For real $x$ we have $$\frac{-c}{1+x^2}\leq G_\Delta^-(x)\leq F_1(x)\leq G_\Delta^+(x)\leq\frac{c}{1+x^2},$$ for some positive constant $c$. Moreover, for any complex number $z$ we have 
$$|G_\Delta^\pm(z)|=O\left(\frac{\Delta^2}{1+\Delta|z|}e^{2\pi\Delta|{\rm Im}\,z|}\right).$$
\item[(ii)] The Fourier transforms of $G_\Delta^\pm$ are continuous functions supported on the interval $[-\Delta,\Delta]$ and satisfy $$|\widehat G_\Delta^\pm(\xi)|=O(1)$$ for all $\xi\in[-\Delta,\Delta]$.
\item[(iii)] The $L^1$-distances of $G_\Delta^\pm$ to $F_1$ are given by 
$$\int_{-\infty}^\infty \left\{F_1(x)-G_\Delta^-(x)\right\}\d x=\frac{2}{\Delta}\int_{1/2}^{3/2}\left\{\log\left(1+e^{-4\pi\Delta(\sigma-1/2)}\right)-\log\left(1+e^{-4\pi\Delta}\right)\right\}\d\sigma$$ 
and 
$$\int_{-\infty}^{\infty}\left\{G_\Delta^+(x)-F_1(x)\right\}\d x=-\frac{2}{\Delta}\int_{1/2}^{3/2}\left\{\log\left(1-e^{-4\pi\Delta(\sigma-1/2)}\right)-\log\left(1-e^{-4\pi\Delta}\right)\right\}\d\sigma.$$
\end{enumerate}
\end{lemma}

\begin{proof}
This is a slightly different version of \cite[Lemma 4]{CCM}. The definitions \eqref{2} and \eqref{5} imply that $F_1(x)=2f_1\left(\frac{x}{2}\right)$, so that one only needs to take $G_\Delta^\pm(z)=2g_{2\Delta}^\pm\left(\frac{z}{2}\right)$ in the notation of \cite[Lemma 4]{CCM}.
\end{proof}

Observe that the $L^1$-distances given in Lemma \ref{lemma3} (iii) are of magnitude $1/\Delta^2$. Indeed,
\begin{align}\label{Evaluation_1}
\begin{split}
\int_{-\infty}^\infty \left\{F_1(x)-G_\Delta^-(x)\right\}\d x & = \frac{1}{2\pi\Delta^2}\int_0^{4\pi\Delta}\left\{\log\left(1+e^{-x}\right)-\log\left(1+e^{-4\pi\Delta}\right)\right\}\d x\\
& \le \frac{1}{2\pi\Delta^2}\int_0^\infty\log\left(1+e^{-x}\right)\dd x\\
& = \frac{1}{2\pi\Delta^2}\int_0^1\frac{\log\left(1+y\right)}{y}\dd y\\
& = \frac{\pi}{24\Delta^2}\,,
\end{split}
\end{align}
because $\frac{\log(1+y)}{y}=\sum_{n=1}^\infty\frac{(-1)^{n-1}}{n}y^{n-1}$ for any $y\in(0,1)$ and the convergence is uniform. Similarly,
\begin{align}\label{Evaluation_2}
\begin{split}
\int_{-\infty}^{\infty}\left\{G_\Delta^+(x)-F_1(x)\right\}\d x &  = \frac{1}{2\pi\Delta^2}\int_0^{4\pi\Delta}\left\{\log\left(1-e^{-x}\right)^{-1}-\log\left(1-e^{-4\pi\Delta}\right)^{-1}\right\}\d x\\
& \le \frac{1}{2\pi\Delta^2}\int_0^\infty\log\left(1-e^{-x}\right)^{-1}\dd x\\
& = \frac{1}{2\pi\Delta^2}\int_0^1\frac{\log\left(1-y\right)^{-1}}{y}\dd y\\
& = \frac{\pi}{12\Delta^2}
\end{split}
\end{align}
because $\frac{\log(1-y)}{y}=-\sum_{n=1}^\infty\frac{1}{n}y^{n-1}$ for any $y\in(0,1)$ and the convergence is monotone.

\smallskip

We are now ready to prove Theorem \ref{theorem1}. The strategy is to apply Lemma \ref{lemma2} to the functions $G_\Delta^-(t-\cdot)$ and $G_\Delta^+(t-\cdot)$, to find bounds for $S_1(t,\pi)$ that depend on $\Delta$ and to optimize the choice of $\Delta$.

\begin{proof}[Proof of Theorem \ref{theorem1}]
We first prove the upper bound. For each $\Delta\ge1$, take $G_\Delta^-$ as in Lemma \ref{lemma3} and let $h(z)=G_\Delta^-(t-z)$. By Lemma \ref{lemma1},
\begin{equation}\label{S1GD}
S_1(t,\pi)\le\frac{1}{\pi}\left(-\sum_\gamma h(\gamma)+\log C(t,\pi)\right)+O(m).
\end{equation}
By Lemma \ref{lemma3} (i), the function $G_\Delta^-(z)(1+z^2)$ is bounded on the real line and $G_\Delta^-(z)=O(\Delta^2e^{2\pi\Delta|{\rm Im}z}|)$. An application of the Phragm\'{e}n-Lindel\"{o}f principle for the function $G_\Delta^-(z)(1+z^2)e^{2\pi\Delta iz}$ tells us that this function is bounded on the upper half plane. Hence $z \mapsto G_\Delta^-(z)(1+z^2)$ is bounded on the strip $0\le{\rm Im}\,z\le \hh + \varepsilon$ (for any $\varepsilon >0$), and since it is real entire, it is bounded on the strip $-\hh - \varepsilon \le{\rm Im}\, z\le\hh + \varepsilon$. Therefore, $h$ satisfies the hypotheses of Lemma \ref{lemma2} and we obtain
\begin{align}\label{x}
\begin{split}
\sum_\gamma h(\gamma)=&\frac{\log N}{2\pi}\int_{-\infty}^\infty h(u)\dd u+\frac{1}{\pi}\sum_{j=1}^m\int_{-\infty}^\infty h(u)\,{\rm Re}\,\frac{\Gamma_\mathbb R'}{\Gamma_\mathbb R}\left(\hh+\mu_j+iu\right)\d u\\
&-\frac{1}{2\pi}\sum_{n=2}^\infty\frac{1}{\sqrt{n}}\left\{\Lambda_{\pi}(n)\,\widehat h\left(\frac{\log n}{2\pi}\right)+\overline{\Lambda_{\pi}(n)}\,\widehat h\left(\frac{-\log n}{2\pi}\right)\right\}\\
&+O(m\Delta^2e^{\pi\Delta}).
\end{split}
\end{align}
For each index $j = 1, 2, \ldots, m$, Stirling's formula \eqref{Stirling} yields 
\begin{align*}
\int_{-\infty}^\infty h(u)\,{\rm Re}\,\frac{\Gamma_\mathbb R'}{\Gamma_\mathbb R}\left(\hh+\mu_j+iu\right)\d u &= \meio\int_{-\infty}^\infty G_\Delta^-(t-u)\log\left|\tfrac{5}{2}+\mu_j+iu\right|\dd u\\
&\ \ \ \ \ \  -\int_{-\infty}^\infty G_\Delta^-(t-u)\,{\rm Re}\left(\frac{1}{\mu_j+\meio+iu}\right)\d u+O(1).
\end{align*}
Combining this with the inequality
\begin{align*}
\left|\int_{-\infty}^\infty G_\Delta^-(t-u)\,\text{Re}\left(\frac{1}{\mu_j+\meio+iu}\right)\d u\right|&\leq c\int_{-\infty}^\infty\left|\text{Re}\left(\frac{1}{\mu_j+\meio+iu}\right)\right|\dd u\\
&= c\int_{-\infty}^\infty\frac{|\text{Re}\,\mu_j+\meio|}{(\text{Re}\,\mu_j+\meio)^2+u^2}\dd u\\
&\leq \pi c
\end{align*}
we find that
\begin{align}\label{9}
\begin{split}
\int_{-\infty}^\infty h(u)\,{\rm Re}\,\frac{\Gamma_\mathbb R'}{\Gamma_\mathbb R}\left(\hh+\mu_j+iu\right)\d u & = \meio\int_{-\infty}^\infty G_\Delta^-(t-u)\log\left|\tfrac{5}{2}+\mu_j+iu\right|\dd u+O(1)\\
&= \meio\int_{-\infty}^{\infty}G_\Delta^-(u)\log\left|\tfrac{5}{2}+\mu_j+it-iu\right|\dd u+O(1)\\
& = \meio\int_{-\infty}^\infty G_\Delta^-(u)\big\{\log(|\mu_j+it|+3)+O(\log(|u|+2))\!\big\}\d u+O(1)\\
& = \meio\log(|\mu_j+it|+3)\int_{-\infty}^\infty G_\Delta^-(u)\dd u+O(1).
\end{split}
\end{align}
By Lemma \ref{lemma3} (ii), the Fourier transform $\widehat h(\xi)=e^{-2\pi it\xi}\,\widehat G_\Delta^-(-\xi)$ is supported on $[-\Delta,\Delta]$ and is uniformly bounded. Also, $|\Lambda_\pi(n)|\le m\Lambda(n)n^{\vartheta}$, and therefore
\begin{align}\label{8}
\frac{1}{2\pi}\sum_{n=2}^\infty\frac{1}{\sqrt{n}}\left\{\Lambda_{\pi}(n)\,\widehat h\left(\frac{\log n}{2\pi}\right)+\overline{\Lambda_{\pi}(n)}\,\widehat h\left(\frac{-\log n}{2\pi}\right)\right\} = O\left(m\sum_{n\leq e^{2\pi\Delta}}\Lambda(n)n^{\vartheta-\meio}\right) = O\left(me^{(1+2\vartheta)\pi\Delta}\right),
\end{align}
where the last equality follows by the Prime Number Theorem and summation by parts.

\smallskip

In view of \eqref{9} and \eqref{8}, equation \eqref{x} becomes
\begin{align*}
\sum_\gamma h(\gamma)=&\frac{\log C(t,\pi)}{2\pi}\int_{-\infty}^\infty G_\Delta^-(u)\dd u+O\left(me^{(1+2\vartheta)\pi\Delta}\right)+O\left(m\Delta^2 e^{\pi\Delta}\right).
\end{align*}
Inserting this in \eqref{S1GD} and using \eqref{Evaluation_1} together with the fact that $\int_{-\infty}^\infty F_1(x)\dd x=2\pi$, we obtain
$$S_1(t,\pi)\le\frac{\log C(t,\pi)}{48\pi\Delta^2}+O(m\Delta^2 e^{(1+2\vartheta)\pi\Delta})$$
for any $t$ and any $\Delta\ge1$. Choosing $$\Delta=\escolha,$$ we arrive at the desired conclusion.

\smallskip

The proof of the lower bound follows the same lines, using $h(z)=G_\Delta^+(t-z)$ and inequality \eqref{Evaluation_2}. 
\end{proof}

\section{Theorem \ref{theorem1} implies Theorem \ref{theorem2}}\label{Sec3}

For real numbers $t <u$ let us denote by $N(t,u, \pi)$ the number of nontrivial zeros of $L(\cdot,\pi)$ with ordinates $\gamma$ such that $t \leq \gamma \leq u$, counted with multiplicity (zeros with ordinates equal to the endpoints $t$ or $u$ are counted with half of their multiplicities). The following fact connects the variation of $S(\cdot,\pi)$ to the nontrivial zeros of $L(\cdot,\pi)$, like equation \eqref{numberofzeros} in the case of $\zeta$. 

\begin{lemma} \label{lemma4}
Let $L(\cdot,\pi)$ satisfy the generalized Riemann hypothesis. Let $t$ and $u$ be real numbers such that $t<u\le t+5$. Then 
\begin{equation}\label{Ntu}
N(t,u, \pi) = S(u,\pi)-S(t,\pi)+\frac{u-t}{2\pi}\, \log C(t,\pi)+O(m).
\end{equation}
\end{lemma}

\begin{proof}
If $v\neq0$ and $v$ is not the imaginary part of a zero of $L(\cdot,\pi)$, then $S'(v,\pi)=\frac{1}{\pi}{\rm Re}\frac{L'}{L}(\meio+iv,\pi)$. At each zero (trivial or non-trivial) $\rho = \hh + i \gamma$ of $L(\cdot, \pi)$ the function $S(\cdot,\pi)$ jumps by the multiplicity of this zero; at $0$ it jumps by $-2r(\pi)$; and for each $j$ such that $-1 < {\rm Re}\,\mu_j <-\meio$, it jumps by $2$ at $-{\rm Im}\,\mu_j$. Therefore
$$N(t,u, \pi)  = S(u,\pi)-S(t,\pi)-\frac{1}{\pi}\int_t^u\text{Re}\,\frac{L'}{L}\left(\hh+iv,\pi\right)\d v+O(m).$$ 
By the definition of $\Lambda(\cdot,\pi)$, $$\frac{\Lambda'}{\Lambda}(s,\pi)=\frac{\log N}{2}+\sum_{j=1}^m\frac{\Gamma_\mathbb R'}{\Gamma_\mathbb R}(s+\mu_j)+\frac{L'}{L}(s,\pi).$$ 
By the functional equation \eqref{Intro_FE}, the real part of $\frac{\Lambda'}{\Lambda}(\cdot,\pi)$ vanishes on the line $\meio+iv$. Therefore
\begin{align*}
-\int_t^u\text{Re}\,\frac{L'}{L}\left(\hh+iv,\pi\right)\d v&=\int_t^u\left\{\frac{\log N}{2}+\sum_{j=1}^m\text{Re}\,\frac{\Gamma_\mathbb R'}{\Gamma_\mathbb R}\left(\mu_j+\hh+iv\right)\right\}\d v\\
&=\frac{1}{2}\int_t^u\left\{\log N+\sum_{j=1}^m\log\left|\tfrac{5}{2}+\mu_j+iv\right|\right\}\d v+O(m)\\
&=\frac{u-t}{2}\log C(t,\pi)+O(m).
\end{align*}
\end{proof}

To derive Theorem \ref{theorem2} from Theorem \ref{theorem1}, we recall the fact that $S_1(\cdot,\pi)$ is a primitive for $S(\cdot,\pi)$. Indeed, for almost every real $v$, $$S(v,\pi)=-\frac{1}{\pi}\int_{1/2}^\infty{\rm Im}\frac{L'}{L}\left(\sigma+iv,\pi\right)\dd\sigma.$$ The function $\frac{L'}{L}(\sigma+iv,\pi)$ is absolutely integrable in the region $\{s\in\mathbb C\,|\,{\rm Re}\,s\geq\meio\text{ and }t\le{\rm Im}\,s\le u\}$ since it has only simple poles and decays exponentially as $\sigma\to\infty$. So we can apply Fubini's theorem to get
\begin{align*}
\int_t^u S(v,\pi)\dd v&=-\frac{1}{\pi}\int_t^u\int_{1/2}^\infty{\rm Im}\frac{L'}{L}\left(\sigma+iv,\pi\right)\d\sigma\dd v\\
&=-\frac{1}{\pi}\int_{1/2}^\infty\int_t^u{\rm Im}\frac{L'}{L}\left(\sigma+iv,\pi\right)\d v\dd\sigma\\
&=\frac{1}{\pi}\int_{1/2}^\infty\big\{\log\left|L\left(\sigma+iu,\pi\right)\right|-\log\left|L\left(\sigma+it,\pi\right)\right|\big\}\dd\sigma\\
&=S_1(u,\pi)-S_1(t,\pi),
\end{align*}
as claimed.

\begin{proof}[Proof of Theorem \ref{theorem2}]
Let $\nu$ be a real number such that $0<\nu\le5$ to be chosen later. The inequality
$$|\log C(t+\nu,\pi)^{3/m}-\log C(t,\pi)^{3/m}|\leq 3\log3$$
implies that
$$\frac{\log C(t+\nu, \pi)^{3/m}}{(\log\log C(t+\nu,\pi)^{3/m})^2}=\frac{\log C(t,\pi)^{3/m}}{(\logfeio)^2}+O(1).$$
Therefore, by Theorem \ref{theorem1} at heights $t$ and $t+\nu$,
\begin{align*}
|S_1(t+\nu,\pi)-S_1(t,\pi)& |\le\frac{(1+2\vartheta)^2\pi}{16}\frac{\log C(t,\pi)}{(\logfeio)^2} +O\left(\frac{\log C(t,\pi)\log\logfeio}{(\logfeio)^3}\right).
\end{align*}
Applying Lemma \ref{lemma4} we see that
\begin{align*}
S_1(t+\nu,\pi)-S_1(t,\pi)& = \int_t^{t+\nu}S(u,\pi)\dd u\\
&\ge\int_t^{t+\nu}\left\{ S(t,\pi)-\frac{u-t}{2\pi}\,\log C(t,\pi)+O(m)\right\}\d u\\
& =  \nu S(t,\pi)-\frac{\nu^2}{4\pi}\,\log C(t,\pi)+O(m)
\end{align*}
and thus
\begin{align*}
S(t,\pi)\le&\frac{(1+2\vartheta)^2\pi}{16\nu}\frac{\log C(t,\pi)}{(\logfeio)^2}+\frac{\nu}{4\pi}\log C(t,\pi) + O\left(\frac{\log C(t,\pi)\log\logfeio}{\nu(\logfeio)^3}\right).
\end{align*}
The upper bound for $S(t,\pi)$ is obtained with the choice 
$$\nu=\frac{(1+2\vartheta)\pi}{2\logfeio}$$ 
(note that $0 < \nu \leq 5$) and the lower bound can be established by the same method, considering $S_1(t,\pi)-S_1(t - \nu,\pi)$.
\end{proof}

\begin{proof}[Proof of Corollary \ref{Cor3}] (i) If $\rho = \hh + i \gamma$ is a zero of $\Lambda(\cdot, \pi)$, part (i) follows directly from Theorem \ref{theorem2} and identity \eqref{Ntu} with $t = \gamma^-$ and $u = \gamma^+$.

\smallskip

\noindent (ii) By Lemma \ref{lemma4}, if $t$ and $u$ are real numbers such that $t<u\le t+5$ and $\Lambda(\cdot,\pi)$ has no zeros between $\meio+it$ and $\meio+iu$, $$\frac{u-t}{2\pi}{\log C(t,\pi)}=-S(u,\pi)+S(t,\pi)+O(m).$$ By Theorem \ref{theorem2}, 
$$u-t\le\frac{(1+2\vartheta)\pi}{\log\log C(t,\pi)^{3/m}}+a\,\frac{\log\log\log C(t,\pi)^{3/m}}{(\log\log C(t,\pi)^{3/m})^2}$$ 
for some universal constant $a$. If $\gamma'-\gamma\le5$, it is enough to let $t\to\gamma$ and $u\to\gamma'$. Otherwise, we let $u=t+5$ and $t\to\gamma$. The obtained inequality is possible only if $\log C(\gamma,\pi)^{3/m}\le e^{\max\{a,3\}}$. Taking $\Lambda(\cdot,\tilde\pi)$ in place of $\Lambda(\cdot,\pi)$, we get $\log C(\gamma',\pi)^{3/m} \le e^{\max\{a,3\}}$. Then 
$$\frac{1}{m}(\log C(\gamma,\pi)+\log C(\gamma',\pi))\le\frac{2}{3}e^{\max\{a,3\}},$$ 
and for some index $j$ we must have
$$\log(|i\gamma+\mu_j|+3)+\log(|i\gamma'+\mu_j|+3)\le\frac{2}{3}e^{\max\{a,3\}}.$$ This implies that $i\gamma+\mu_j$ and $i\gamma'+\mu_j$ are bounded by some universal constant.
\end{proof}

\section*{Acknowledgements}
\noindent \noindent E.C. acknowledges support from CNPq-Brazil grants $302809/2011-2$ and $477218/2013-0$, and FAPERJ grant $E-26/103.010/2012$. R.F. is indebted to CNPq-Brazil for its financial support.

\end{document}